\documentclass[12pt, reqno]{amsart}
\makeatletter
\@namedef{subjclassname@1991}{$\mathrm{1991}$ Mathematics Subject Classification}
\@namedef{subjclassname@2000}{$\mathrm{2000}$ Mathematics Subject Classification}
\@namedef{subjclassname@2010}{$\mathrm{2010}$ Mathematics Subject Classification}
\@namedef{subjclassname@2020}{$\mathrm{2020}$ Mathematics Subject Classification}
\makeatother
\usepackage{amsmath,amsthm, amscd, amsfonts, amssymb, graphicx, color}
\usepackage[bookmarksnumbered, colorlinks, plainpages,linkcolor=blue,urlcolor=blue,citecolor=blue]{hyperref}
\usepackage{amsmath}
\usepackage{amssymb}
\newtheorem{thm}{Theorem}[section]
\newtheorem{cor}[thm]{Corollary}
\newtheorem{lem}[thm]{Lemma}

\newtheorem{exam}[thm]{Example}
\numberwithin{equation}{section}

\begin{document}

\title{The generalized core-EP inverse of operator matrices with generalized Schur complements}

\author{Huanyin Chen}
\address{School of Big Data, Fuzhou University of International Studies and Trade, Fuzhou 350202, China}
\email{<huanyinchenfz@163.com>}

\subjclass[2020]{16U50, 15A09, 46H05.}\keywords{core inverse; core-EP inverse; generalized core-EP inverse; block operator matrix; Banach algebra.}

\begin{abstract} In this paper, we investigate the generalized core-EP inverse of block operator matrices over a Banach algebra.
As a result, novel representations and explicit formulae for the core and core-EP inverses of the $2\times 2$ block operator matrices are established.
\end{abstract}

\maketitle

\section{Introduction}

A Banach $*$-algebra is a Banach algebra $\mathcal{A}$ endowed with an involution $*: \mathcal{A} \to \mathcal{A}$ satisfying $(x^*)^* = x$, $(x+y)^* = x^* + y^*$, $(\lambda x)^* = \bar{\lambda}x^*$, $(xy)^* = y^*x^*$, and $\|x^*\| = \|x\|$ for all $x, y \in \mathcal{A}$ and $\lambda \in \mathbb{C}$.
Let $\mathcal{A}$ be a complex Banach algebra and let $a \in \mathcal{A}$. An element $x \in \mathcal{A}$ is called the group inverse of $a$ (denoted $a^\#$) if it satisfies: $ax^2=x, ax=xa, xa^2=a.$ An element $x \in \mathcal{A}$ is the Drazin inverse of $a$ (denoted $a^D$) if there exists a non-negative integer $k$ such that $ax^2=x, ax=xa, xa^{k+1}=a^k.$ An element $x \in \mathcal{A}$ is the generalized Drazin inverse of $a$ (denoted $a^d$) if
$ax^2=x, ax=xa, a-xa^2\in \mathcal{A}^{\mathrm{qnil}}$. Here,
$\mathcal{A}^{\mathrm{qnil}}$ denotes the set of quasinilpotent elements of $\mathcal{A}$. As is well known,
\[
\mathcal{A}^{\mathrm{qnil}} = \left\{ x \in \mathcal{A} \mid \lim_{n\to \infty} \|x^n\|^{\frac{1}{n}} = 0 \right\}.
\] These generalized inverses play important roles in matrix and operator theory (see~\cite{C,C1,CV,CD,D,D1,MD,W,ZM}).

Let $\mathcal{A}$ be a Banach *-algebra. An element $x \in \mathcal{A}$ is the core inverse of $a$ (denoted $a^{\tiny\textcircled{\#}}$) if $$ax^2=x, (ax)^*=ax, xa^2=a.$$ An element $x \in \mathcal{A}$ is the core-EP inverse of $a$ (denoted by $a^{\tiny\textcircled{D}}$) if there exists a non-negative integer $k$ such that $$ax^2=x, (ax)^*=ax, xa^{k+1}=a^k$$ (see~\cite{GC}). Following Chen and Sheibani, an element $x \in \mathcal{A}$ is the generalized core-EP inverse of $a$ (denoted $a^{\tiny\textcircled{d}}$) if
$$ax^2=x, (ax)^*=ax, \lim\limits_{n\to \infty}||a^n-xa^{n+1}||^{\frac{1}{n}}=0$$ (se~\cite{C3,C4}). This generalizes the core-EP inverse from matrices to infinite-dimensional Banach space operators.

The core and core-EP inverses of the matrices have attracted considerable attention. Xu ~\cite{XS} considered the core inverse of a special form of triangular matrices, while Gao and Chen ~\cite{GC2} provided investigations of the core-EP inverse of a triangular complex matrix. It is worth noting that a triangular matrix may fail to have a core-EP inverse even when its diagonal entries do ~\cite[Example 1]{GC2}, highlighting the complexity of this problem. Recent effort have also explored the core-EP inverse of companion matrices ~\cite{GC3}. In ~\cite{C5}, Chen and Sheibani investigated the generalized core-EP inverse for triangular matrices over a $C^*$-algebra. For further results on these generalized inverses, we refer the reader to ~\cite{??}.

Let $$x=\left(
\begin{array}{cc}
a&b\\
c&d
\end{array}
\right)\in M_2(\mathcal{A}),$$ where $a\in \mathcal{A}^{\tiny\textcircled{d}}$, and let $s=d-ca^{\tiny\textcircled{d}}b\in \mathcal{A}^{\tiny\textcircled{d}}$ be the generalized Schur complement of $a$ in $x$.
This aim of this paper is to investigate the generalized core-EP inverse for several novel classes of the operator matrix $x$ with the generalized Schur complement over a Banach algebra. Consequently, new representations of the core and core-EP inverses of some novel block operator matrices are derived.

Throughout this paper, all Banach $*$-algebras are assumed to be complex and unital. We denote by $A^{D}, A^d, A^{\tiny\textcircled{\#}}, A^{\tiny\textcircled{D}}$, and $A^{\tiny\textcircled{d}}$ the sets of all Drazin invertible, generalized Drazin invertible, core invertible, core-EP invertible, and generalized core-EP invertible elements in $\mathcal{A}$, respectively. Let $\mathbb{C}^{n \times n}$ denote the Banach algebra of all $n \times n$ complex matrices, equipped with the conjugate transpose as its involution. For an element $a \in \mathcal{A}$ possessing a generalized Drazin inverse, we denote its spectral idempotent $1-aa^d$ by $a^{\pi}$.

\section{key lemmas}

For subsequent use, we require the following lemmas.

\begin{lem} Let $a\in \mathcal{A}^{\tiny\textcircled{d}}$ and $b\in \mathcal{A}^{qnil}$. If $ba=0$, then
$a+b\in \mathcal{A}^{\tiny\textcircled{d}}$. In this case, $$(a+b)^{\tiny\textcircled{d}}=a^{\tiny\textcircled{d}}.$$\end{lem}
\begin{proof} By virtue of ~\cite[Theorem 2.1]{C3}, there exists a decomposition $$a=x+y, yx=0, x\in \mathcal{A}^{\tiny\textcircled{\#}}, y\in \mathcal{A}^{qnil}.$$ Evidently, $x=aa^{\tiny\textcircled{d}}a$ and $y=a-aa^{\tiny\textcircled{d}}a$. Since $ba=0$, we see that
$by=0$. In view of~\cite[Lemma 2.1]{CD}, $b+y\in \mathcal{A}^{qnil}$. Moreover, we see that $(b+y)x=bx+yx=0$. According to~\cite[Theorem 2.1]{C3},
$a+b=x+(b+y)\in \mathcal{A}^{\tiny\textcircled{d}}$. In this case, $(a+b)^{\tiny\textcircled{d}}=x^{\tiny\textcircled{\#}}=a^{\tiny\textcircled{d}},$ as asserted.\end{proof}

\begin{lem} Let $a\in \mathcal{A}^{\tiny\textcircled{d}}$ and $b\in \mathcal{A}$. Then the following are equivalent:\end{lem}
\begin{enumerate}
\vspace{-.5mm}
\item [(1)] $a^{\pi}b=0$.
\vspace{-.5mm}
\item [(2)] $a^{\tau}b=0$.
\vspace{-.5mm}
\item [(3)] $a^{\sigma}b=0$.
\end{enumerate}
\begin{proof} $(1)\Rightarrow (2)$ Since $a^{\pi}b=0$, we have that $b=aa^db$. Then $a^{\tiny\textcircled{d}}ab=a^{\tiny\textcircled{d}}a^2a^db=aa^db=b$.
Hence $(1-a^{\tiny\textcircled{d}}a)b=0$, as required.

$(2)\Rightarrow (3)$ Since $(1-a^{\tiny\textcircled{d}}a)b=0$, we have $b=a^{\tiny\textcircled{d}}ab$. Then
$(1-aa^{\tiny\textcircled{d}})b=(1-aa^{\tiny\textcircled{d}})a^{\tiny\textcircled{d}}ab=0.$

$(3)\Rightarrow (1)$ Since $(1-aa^{\tiny\textcircled{d}})b=0$, we have $b=aa^{\tiny\textcircled{d}}b$. Then
$$aa^db=a^2a^da^{\tiny\textcircled{d}}b=aa^{\tiny\textcircled{d}}b=b.$$ Therefore $a^{\pi}b=0$, as asserted.
\end{proof}

\begin{lem} Let $a\in \mathcal{A}^{\tiny\textcircled{\#}}$ and $b\in \mathcal{A}^{\tiny\textcircled{d}}$. If $ba=0$ and $b^*a=0$, then
$a+b\in \mathcal{A}^{\tiny\textcircled{d}}$. In this case, $$(a+b)^{\tiny\textcircled{d}}=a^{\tiny\textcircled{\#}}
+a^{\pi}b^{\tiny\textcircled{d}}.$$\end{lem}
\begin{proof} By virtue of ~\cite[Theorem 2.1]{C3}, there exists a decomposition $$b=x+y, yx=0, x\in \mathcal{A}^{\tiny\textcircled{\#}}, y\in \mathcal{A}^{qnil}.$$ Evidently, $x=bb^{\tiny\textcircled{d}}b$ and $y=b-bb^{\tiny\textcircled{d}}b$. Since $ba=0$ and $b^*a=0$, we see that
$xa=0$ and $x^*a=(bb^{\tiny\textcircled{d}}b)^*a=(b^{\tiny\textcircled{d}}b)^*(b^*a)=0$. By virtue of ~\cite[Theorem 4.3]{XS1}, $a+x\in \mathcal{A}^{\tiny\textcircled{\#}}$. We easily check that $y(a+x)=ya+yx=(1-bb^{\tiny\textcircled{d}})ba+yx=0$. By virtue of
~\cite[Theorem 2.1]{C3}, $a+b=(a+x)+y\in \mathcal{A}^{\tiny\textcircled{d}}$. In this case, $(a+b)^{\tiny\textcircled{d}}=(a+x)^{\tiny\textcircled{\#}}=a^{\pi}x^{\tiny\textcircled{\#}}+a^{\tiny\textcircled{\#}}=a^{\tiny\textcircled{\#}}
+a^{\pi}b^{\tiny\textcircled{d}},$ as required.\end{proof}

\begin{lem} Let $a\in \mathcal{A}$. Then $a\in \mathcal{A}^{\tiny\textcircled{\#}}$ if and only if
$a\in \mathcal{A}^{\tiny\textcircled{d}}$ and $aa^{\tiny\textcircled{d}}a=a.$\end{lem}
\begin{proof} $\Longrightarrow$ Since $a\in \mathcal{A}^{\tiny\textcircled{\#}}$, we have $a^{\tiny\textcircled{d}}=
a^{\tiny\textcircled{\#}}$, and then $aa^{\tiny\textcircled{d}}a=a.$

$\Longrightarrow$ For any $n\in {\Bbb N}$, we deduce that
$$\begin{array}{rll}
||a-a^{\tiny\textcircled{d}}a^2||&=&||aa^{\tiny\textcircled{d}}a-a^{\tiny\textcircled{d}}a(aa^{\tiny\textcircled{d}}a)||\\
&=&||a^n(a^{\tiny\textcircled{d}})^na-a^{\tiny\textcircled{d}}a\big(a^n(a^{\tiny\textcircled{d}})^n)a\big)||\\
&=&||\big(a^n-a^{\tiny\textcircled{d}}a^{n+1}\big)(a^{\tiny\textcircled{d}})^n)a||\\
&\leq &||a^n-a^{\tiny\textcircled{d}}a^{n+1}||||(a^{\tiny\textcircled{d}})^na||.
\end{array}$$ Then
$$||a-a^{\tiny\textcircled{d}}a^2||^{\frac{1}{n}}\leq ||a^n-a^{\tiny\textcircled{d}}a^{n+1}|^{\frac{1}{n}}||a^{\tiny\textcircled{d}}||||a||^{\frac{1}{n}}.$$
Since $\lim\limits_{n\to \infty}||a^n-a^{\tiny\textcircled{d}}a^{n+1}||^{\frac{1}{n}}=0,$ we have $\lim\limits_{n\to \infty}||a-a^{\tiny\textcircled{d}}a^2||^{\frac{1}{n}}=0.$ This implies that $a-a^{\tiny\textcircled{d}}a^2=0$.
Thus we  have $a=a(a^{\tiny\textcircled{d}})^2, a=a^{\tiny\textcircled{d}}a^2$ and $(aa^{\tiny\textcircled{d}})^*=aa^{\tiny\textcircled{d}}$.
Accordingly, $a^{\tiny\textcircled{\#}}=a^{\tiny\textcircled{d}}$, as desired.\end{proof}

\section{The main results}

In this section, we consider
$$x=\left(
\begin{array}{cc}
a&b\\
c&d
\end{array}
\right)\in M_2(\mathcal{A}) \quad \quad (1),$$ where $a$ and $s$ admit generalized core-EP inverses. Set $a^{\sigma}=1-aa^{\tiny\textcircled{d}}$ and $a^{\tau}=1-a^{\tiny\textcircled{d}}a.$ Set $s=d-ca^{\tiny\textcircled{d}}b$ and $p=\left(
\begin{array}{cc}
a^{\tau}&0\\
0&s^{\tau}
\end{array}
\right).$ We now derive

\begin{thm} Let $x$ be defined as in $(1)$. If $$a^{\pi}bs^{\tiny\textcircled{d}}=s^{\pi}ca^{\tiny\textcircled{d}}=a^{\tiny\textcircled{d}}bs^{\tau}=
s^{\tiny\textcircled{d}}ca^{\tau}=0$$ and
$\lim\limits_{n\to \infty}||px^n||^{\frac{1}{n}}=0,$ then $x$ has generalized core-EP inverse. In this case, $$x^{\tiny\textcircled{d}}=\left(
\begin{array}{cc}
a^{\tiny\textcircled{d}}+a^{\tiny\textcircled{d}}bs^{\tiny\textcircled{d}}ca^{\tiny\textcircled{d}}&-a^{\tiny\textcircled{d}}bs^{\tiny\textcircled{d}}\\
-s^{\tiny\textcircled{d}}ca^{\tiny\textcircled{d}}&s^{\tiny\textcircled{d}}
\end{array}
\right).$$\end{thm}
\begin{proof} Set $z=\left(
\begin{array}{cc}
a^{\tiny\textcircled{d}}+a^{\tiny\textcircled{d}}bs^{\tiny\textcircled{d}}ca^{\tiny\textcircled{d}}&-a^{\tiny\textcircled{d}}bs^{\tiny\textcircled{d}}\\
-s^{\tiny\textcircled{d}}ca^{\tiny\textcircled{d}}&s^{\tiny\textcircled{d}}
\end{array}
\right).$ Then we verify that
$$\begin{array}{rll}
xz&=&\left(
\begin{array}{cc}
a&b\\
c&d
\end{array}
\right)\left(
\begin{array}{cc}
a^{\tiny\textcircled{d}}+a^{\tiny\textcircled{d}}bs^{\tiny\textcircled{d}}ca^{\tiny\textcircled{d}}&-a^{\tiny\textcircled{d}}bs^{\tiny\textcircled{d}}\\
-s^{\tiny\textcircled{d}}ca^{\tiny\textcircled{d}}&s^{\tiny\textcircled{d}}
\end{array}
\right)\\
&=&\left(
\begin{array}{cc}
aa^{\tiny\textcircled{d}}&0\\
0&ss^{\tiny\textcircled{d}}
\end{array}
\right),\\
(xz)^*&=&xz,\\
zx&=&\left(
\begin{array}{cc}
a^{\tiny\textcircled{d}}+a^{\tiny\textcircled{d}}bs^{\tiny\textcircled{d}}ca^{\tiny\textcircled{d}}&-a^{\tiny\textcircled{d}}bs^{\tiny\textcircled{d}}\\
-s^{\tiny\textcircled{d}}ca^{\tiny\textcircled{d}}&s^{\tiny\textcircled{d}}
\end{array}
\right)\left(
\begin{array}{cc}
a&b\\
c&d
\end{array}
\right)\\
&=&\left(
\begin{array}{cc}
a^{\tiny\textcircled{d}}a&0\\
0&s^{\tiny\textcircled{d}}s
\end{array}
\right),\\
xz^2&=&(xz)z=\left(
\begin{array}{cc}
aa^{\tiny\textcircled{d}}&0\\
0&ss^{\tiny\textcircled{d}}
\end{array}
\right)\left(
\begin{array}{cc}
a^{\tiny\textcircled{d}}+a^{\tiny\textcircled{d}}bs^{\tiny\textcircled{d}}ca^{\tiny\textcircled{d}}&-a^{\tiny\textcircled{d}}bs^{\tiny\textcircled{d}}\\
-s^{\tiny\textcircled{d}}ca^{\tiny\textcircled{d}}&s^{\tiny\textcircled{d}}
\end{array}
\right)\\
&=&\left(
\begin{array}{cc}
a^{\tiny\textcircled{d}}+a^{\tiny\textcircled{d}}bs^{\tiny\textcircled{d}}ca^{\tiny\textcircled{d}}&-a^{\tiny\textcircled{d}}bs^{\tiny\textcircled{d}}\\
-s^{\tiny\textcircled{d}}ca^{\tiny\textcircled{d}}&s^{\tiny\textcircled{d}}
\end{array}
\right)=z,\\
(1-zx)x^n&=&px^n.
\end{array}$$ This completes the proof.\end{proof}

\begin{cor} Let $x$ be defined as in $(1)$. If $$a^{\pi}b=s^{\pi}c=bs^{\tau}=
ca^{\tau}=0,$$ then $x$ has generalized core-EP inverse. In this case, $$x^{\tiny\textcircled{d}}=\left(
\begin{array}{cc}
a^{\tiny\textcircled{d}}+a^{\tiny\textcircled{d}}bs^{\tiny\textcircled{d}}ca^{\tiny\textcircled{d}}&-a^{\tiny\textcircled{d}}bs^{\tiny\textcircled{d}}\\
-s^{\tiny\textcircled{d}}ca^{\tiny\textcircled{d}}&s^{\tiny\textcircled{d}}
\end{array}
\right)$$\end{cor}
\begin{proof} As $a^{\pi}b=0$, it follows by Lemma 2.2 that $a^{\tau}b=0$. Moreover, we have
$(a^{\tau}a)^nb=0$. We directly verify that
$s^{\tau}d=s^{\tau}(s+ca^{\tiny\textcircled{d}}b)=s^{\tau}s$.
Then
$$px^n=\left(
\begin{array}{cc}
a^{\tau}&0\\
0&s^{\tau}
\end{array}
\right)\left(
\begin{array}{cc}
a&b\\
c&d
\end{array}
\right)^n=\left(
\begin{array}{cc}
a^{\tau}a^n&0\\
0&s^{\tau}s^n
\end{array}
\right).$$ Therefore $\lim\limits_{n\to \infty}||px^n||^{\frac{1}{n}}=0.$ This completes the prof by Theorem 3.1.
\end{proof}

\begin{cor} Let $x$ be defined as in $(1)$. If $s:=d-ca^{\tiny\textcircled{d}}b\in \mathcal{A}^{-1}$ and $$a^{\pi}b=0~\mbox{and} ~ca^{\tau}=0,$$ then $x$ has a generalized core-EP inverse. In this case, $$x^{\tiny\textcircled{d}}=\left(
\begin{array}{cc}
a^{\tiny\textcircled{d}}+a^{\tiny\textcircled{d}}bs^{-1}ca^{\tiny\textcircled{d}}&-a^{\tiny\textcircled{d}}bs^{-1}\\
-s^{-1}ca^{\tiny\textcircled{d}}&s^{-1}
\end{array}
\right).$$
\end{cor}
\begin{proof} Since $s=d-ca^{\tiny\textcircled{d}}b\in \mathcal{A}^{-1}$, we see that $s^{\tiny\textcircled{d}}=s^{-1}$ and $s^{\pi}=s^{\tau}=0$.
Therefore we complete the proof by Corollary 3.2.\end{proof}

We next consider
$$x=\left(
\begin{array}{cc}
a&b\\
c&d
\end{array}
\right)\in M_2(\mathcal{A}) \quad \quad (2),$$ where $a$ and $s=d-ca^{\tiny\textcircled{D}}b$ admit core-EP inverses. We are ready to prove:

\begin{thm} Let $x$ be defined as in $(2)$. If $$a^{\pi}bs^{\tiny\textcircled{D}}=s^{\pi}ca^{\tiny\textcircled{D}}=a^{\tiny\textcircled{D}}bs^{\tau}=
s^{\tiny\textcircled{D}}ca^{\tau}=0$$ and
$px^n=0$ for some $n\in {\Bbb N}$, then $x$ has core-EP inverse. In this case,
$$x^{\tiny\textcircled{D}}=\left(
\begin{array}{cc}
a^{\tiny\textcircled{D}}+a^{\tiny\textcircled{D}}bs^{\tiny\textcircled{D}}ca^{\tiny\textcircled{D}}&-a^{\tiny\textcircled{D}}bs^{\tiny\textcircled{D}}\\
-s^{\tiny\textcircled{D}}ca^{\tiny\textcircled{D}}&s^{\tiny\textcircled{D}}
\end{array}
\right).$$\end{thm}
\begin{proof} In view of Theorem 3.1, $x$ has generalized core-EP inverse and
$$x^{\tiny\textcircled{d}}=\left(
\begin{array}{cc}
a^{\tiny\textcircled{D}}+a^{\tiny\textcircled{D}}bs^{\tiny\textcircled{D}}ca^{\tiny\textcircled{D}}&-a^{\tiny\textcircled{D}}bs^{\tiny\textcircled{D}}\\
-s^{\tiny\textcircled{D}}ca^{\tiny\textcircled{D}}&s^{\tiny\textcircled{D}}
\end{array}
\right).$$ As in the proof of Theorem 3.1, we have $$\begin{array}{rll}
(1-xx^{\tiny\textcircled{d}})x^n&=&\left(
\begin{array}{cc}
1-aa^{\tiny\textcircled{D}}&0\\
0&1-ss^{\tiny\textcircled{D}}
\end{array}
\right)x^n\\
&=&px^n=0.
\end{array}$$ Thus, $(1-x\alpha)x\in M_2(\mathcal{A})^{nil}$.
Accordingly, $x$ has a core-EP inverse and $x^{\tiny\textcircled{D}}=x^{\tiny\textcircled{d}}$, as desired.\end{proof}

\begin{cor} Let $x$ be defined as in $(2)$. If $s:=d-ca^{\tiny\textcircled{D}}b\in \mathcal{A}^{-1}$ and $$a^{\pi}b=0~\mbox{and} ~ca^{\tau}=0,$$ then $x$ has a generalized core-EP inverse. In this case, $$x^{\tiny\textcircled{D}}=\left(
\begin{array}{cc}
a^{\tiny\textcircled{D}}+a^{\tiny\textcircled{D}}bs^{-1}ca^{\tiny\textcircled{D}}&-a^{\tiny\textcircled{D}}bs^{-1}\\
-s^{-1}ca^{\tiny\textcircled{D}}&s^{-1}
\end{array}
\right).$$
\end{cor}
\begin{proof} This is obvious by Theorem 3.4.\end{proof}

We now turn to the proof of the following result.

\begin{thm} Let $x$ be defined as in $(1)$. If $$a^{\pi}bs^{\tiny\textcircled{d}}=s^{\pi}ca^{\tiny\textcircled{d}}=a^{\tiny\textcircled{d}}bs^{\tau}=
s^{\tiny\textcircled{d}}ca^{\tau}=0,$$ and
then $x$ has core inverse if and only if $a,s\in \mathcal{A}^{\tiny\textcircled{\#}}$ and $a^{\pi}b=s^{\pi}c=0$.
In this case,
$$x^{\tiny\textcircled{d}}=\left(
\begin{array}{cc}
a^{\tiny\textcircled{\#}}+a^{\tiny\textcircled{\#}}bs^{\tiny\textcircled{\#}}ca^{\tiny\textcircled{\#}}&-a^{\tiny\textcircled{\#}}bs^{\tiny\textcircled{\#}}\\
-s^{\tiny\textcircled{\#}}ca^{\tiny\textcircled{\#}}&s^{\tiny\textcircled{\#}}\end{array}
\right).$$\end{thm}
\begin{proof} $\Longrightarrow $ Since $x$ is core inverse, we see that
$p=1-xx^{\tiny\textcircled{\#}}$ and then $px=0$. By virtue of Theorem 3.1, $$x^{\tiny\textcircled{\#}}=\left(
\begin{array}{cc}
a^{\tiny\textcircled{d}}+a^{\tiny\textcircled{d}}bs^{\tiny\textcircled{d}}ca^{\tiny\textcircled{d}}&-a^{\tiny\textcircled{d}}bs^{\tiny\textcircled{d}}\\
-s^{\tiny\textcircled{d}}ca^{\tiny\textcircled{d}}&s^{\tiny\textcircled{d}}
\end{array}
\right).$$ As in the proof of Theorem 3.1, we see that
$$1-xx^{\tiny\textcircled{\#}}=\left(
\begin{array}{cc}
1-aa^{\tiny\textcircled{d}}&0\\
0&1-ss^{\tiny\textcircled{d}}
\end{array}
\right).$$ Since $(1-xx^{\tiny\textcircled{\#}})x=0$, we deduce that $(1-aa^{\tiny\textcircled{d}})a=0, (1-aa^{\tiny\textcircled{d}})b=0,
(1-ss^{\tiny\textcircled{d}})c=0$ and $(1-ss^{\tiny\textcircled{d}})d=0$. This implies that
$a,s\in \mathcal{A}^{\tiny\textcircled{\#}}$ and $a^{\pi}b=s^{\pi}c=0$.

$\Longleftarrow$ Let $z=\left(
\begin{array}{cc}
a^{\tiny\textcircled{d}}+a^{\tiny\textcircled{d}}bs^{\tiny\textcircled{d}}ca^{\tiny\textcircled{d}}&-a^{\tiny\textcircled{d}}bs^{\tiny\textcircled{d}}\\
-s^{\tiny\textcircled{d}}ca^{\tiny\textcircled{d}}&s^{\tiny\textcircled{d}}
\end{array}
\right).$ As in the proof of Theorem 3.1, we check that

$$xz^2=z, (xz)^*=xz$$. Moreover, we see that $zx=\left(
\begin{array}{cc}
a^{\tiny\textcircled{d}}a&0\\
0&s^{\tiny\textcircled{d}}s
\end{array}
\right).$ Then $(1-zx)x=0$. This implies that
$x^{\tiny\textcircled{\#}}=z$, as desired.\end{proof}

\begin{cor} Let $x$ be defined as in $(1)$ and assume that $a,s\in \mathcal{A}^{\tiny\textcircled{\#}}$. If
$$a^{\pi}b=s^{\pi}c=a^{\pi}c^*s=s^{\pi}b^*a=0,$$
then $x$ has core inverse.
In this case, $$x^{\tiny\textcircled{\#}}=\left(
\begin{array}{cc}
a^{\tiny\textcircled{\#}}+a^{\tiny\textcircled{\#}}bs^{\tiny\textcircled{\#}}ca^{\tiny\textcircled{\#}}&-a^{\tiny\textcircled{\#}}bs^{\tiny\textcircled{\#}}\\
-s^{\tiny\textcircled{\#}}ca^{\tiny\textcircled{\#}}&s^{\tiny\textcircled{\#}}\end{array}
\right).$$\end{cor}
\begin{proof} Straightforward by Theorem 3.6.\end{proof}

We are ready to prove:

\begin{thm} Let $x$ be defined as in $(1)$. If
$$a^{\pi}bs^{\tiny\textcircled{d}}=a^{\pi}bc=s^{\pi}ca^{\tiny\textcircled{d}}=ss^{\sigma}c=0,$$ then $x$ has generalized core-EP inverse. In this case,
$$x^{\tiny\textcircled{d}}=\left(
\begin{array}{cc}
a^{\tiny\textcircled{d}}+a^{\tiny\textcircled{d}}bs^{\tiny\textcircled{d}}ca^{\tiny\textcircled{d}}&-a^{\tiny\textcircled{d}}bs^{\tiny\textcircled{d}}\\
-s^{\tiny\textcircled{d}}ca^{\tiny\textcircled{d}}&s^{\tiny\textcircled{d}}
\end{array}
\right).$$\end{thm}
\begin{proof} Obviously, we have a decomposition $x=y+z$, where
$$y=\left(
\begin{array}{cc}
aa^{\sigma}&bs^{\sigma}\\
ca^{\sigma}&ds^{\sigma}
\end{array}
\right), z=\left(
\begin{array}{cc}
a^2a^{\tiny\textcircled{d}}&bss^{\tiny\textcircled{d}}\\
caa^{\tiny\textcircled{d}}&dss^{\tiny\textcircled{d}}
\end{array}
\right).$$

Step 1. $yz=0$. As $s^{\pi}ca^{\tiny\textcircled{d}}=0$, it follows by Lemma 2.2 that $s^{\sigma}ca^{\tiny\textcircled{d}}=0$.
Then $s^{\sigma}caa^{\tiny\textcircled{d}}=[s^{\sigma}ca^{\tiny\textcircled{d}}]a^2a^{\tiny\textcircled{d}}=0$. Similarly, we have
$(aa^{\sigma})(bss^{\tiny\textcircled{d}})=[aa^{\sigma}bs^{\tiny\textcircled{d}}]s^2s^{\tiny\textcircled{d}}=0$. Since $d=s+ca^{\tiny\textcircled{d}}b$, we deduce that $bs^{\sigma}dss^{\tiny\textcircled{d}}=bs^{\sigma}[s+ca^{\tiny\textcircled{d}}b]ss^{\tiny\textcircled{d}}=
b(s^{\sigma}ca^{\tiny\textcircled{d}})bss^{\tiny\textcircled{d}}=0$. Furthermore, we have $[ds^{\sigma}][dss^{\tiny\textcircled{d}}]=
ds^{\sigma}[s+ca^{\tiny\textcircled{d}}b]ss^{\tiny\textcircled{d}}=d[s^{\sigma}ca^{\tiny\textcircled{d}}]bss^{\tiny\textcircled{d}}=0$.
Therefore we verify that $yz=0$.

Step 2. $y\in M_2(\mathcal{A})^{qnil}$. We observe that
$y=y_1+y_2+y_3$, where
$$y_1=\left(
\begin{array}{cc}
aa^{\sigma}&a^{\sigma}bs^{\sigma}\\
0&ss^{\sigma}
\end{array}
\right), y_2=\left(
\begin{array}{cc}
0&0\\
s^{\sigma}ca^{\sigma}&0
\end{array}
\right), y_3=\left(
\begin{array}{cc}
0&aa^{\tiny\textcircled{d}}bs^{\sigma}\\
ss^{\tiny\textcircled{d}}ca^{\sigma}&ss^{\tiny\textcircled{d}}ds^{\sigma}
\end{array}
\right).$$ By hypothesis, we verify that $y_1y_2=y_1y_3=y_2y_3=0$. As $y_1,y_2$ are quasinilpotent, it follows by~\cite[Lemma 2.1]{CD} that
$y_1+y_2$ is quasinilpotent. Obviously,
$y_3^2=\left(
\begin{array}{cc}
0&aa^{\tiny\textcircled{d}}bs^{\sigma}\\
ss^{\tiny\textcircled{d}}ca^{\sigma}&ss^{\tiny\textcircled{d}}ds^{\sigma}
\end{array}
\right)^2=0$. Since $(y_1+y_2)y_3=y_1y_3+y_2y_3=0$, by ~\cite[Lemma 2.1]{CD}, we see that $y_1+y_2+y_3$ is qusinilpotent, as required.

Step 3. $z\in M_2(\mathcal{A})^{\tiny\textcircled{\#}}$.

Obviously, $\big(a^2a^{\tiny\textcircled{d}}\big)^{\tiny\textcircled{\#}}=a^{\tiny\textcircled{d}}.$
Then $\big(a^2a^{\tiny\textcircled{d}}\big)^{\sigma}=1-\big(a^2a^{\tiny\textcircled{d}}\big)a^{\tiny\textcircled{d}}=
1-aa^{\tiny\textcircled{d}}=a^{\sigma}$. Moreover, we have
$$\begin{array}{rll}
S&:=&dss^{\tiny\textcircled{d}}-caa^{\tiny\textcircled{d}}\big(a^2a^{\tiny\textcircled{d}}\big)^{\tiny\textcircled{d}}bss^{\tiny\textcircled{d}}\\
&=&dss^{\tiny\textcircled{d}}-ca(a^{\tiny\textcircled{d}})^2bss^{\tiny\textcircled{d}}\\
&=&\big(d-ca^{\tiny\textcircled{d}}b\big)ss^{\tiny\textcircled{d}}\\
&=&s^2s^{\tiny\textcircled{d}}.
\end{array}$$ Hence, $S\in \mathcal{A}^{\tiny\textcircled{\#}}$ and $S^{\tiny\textcircled{\#}}=s^{\tiny\textcircled{d}}.$
Then $S^{\sigma}=1-SS^{\tiny\textcircled{\#}}=1-ss^{\tiny\textcircled{d}}=s^{\sigma}$.
We verify that
$$\begin{array}{rll}
\big(a^2a^{\tiny\textcircled{d}}\big)^{\pi}\big(bss^{\tiny\textcircled{d}}\big)&=&a^{\pi}bss^{\tiny\textcircled{d}}=0,\\
\big(caa^{\tiny\textcircled{d}}\big)\big(a^2a^{\tiny\textcircled{d}}\big)^{\sigma}&=&\big(caa^{\tiny\textcircled{d}}\big)a^{\sigma}=0,\\
\big(bss^{\tiny\textcircled{d}}\big)S^{\sigma}&=&\big(bss^{\tiny\textcircled{d}}\big)s^{\sigma}=0.\\
S^{\pi}\big(caa^{\tiny\textcircled{d}}\big)&=&s^{\pi}caa^{\tiny\textcircled{d}}=0.\end{array}$$
By virtue of Corollary 3.7, $z\in M_2(\mathcal{A})^{\tiny\textcircled{\#}}$ and

$$z^{\tiny\textcircled{\#}}=\left(
\begin{array}{cc}
a^{\tiny\textcircled{d}}+a^{\tiny\textcircled{d}}bs^{\tiny\textcircled{d}}ca^{\tiny\textcircled{d}}&-a^{\tiny\textcircled{d}}bs^{\tiny\textcircled{d}}\\
-s^{\tiny\textcircled{d}}ca^{\tiny\textcircled{d}}&s^{\tiny\textcircled{d}}\end{array}
\right).$$
By virtue of ~\cite[Theorem 2.1]{C3}, $x^{\tiny\textcircled{d}}=z^{\tiny\textcircled{\#}}$, as desired.\end{proof}

\begin{cor} Let $x$ be defined as in $(2)$. If $$a^{\pi}bs^{\tiny\textcircled{D}}=a^{\pi}bc=s^{\pi}ca^{\tiny\textcircled{D}}=ss^{\sigma}c=0,$$ then $x$ has core-EP inverse. In this case,
$$x^{\tiny\textcircled{D}}=\left(
\begin{array}{cc}
a^{\tiny\textcircled{D}}+a^{\tiny\textcircled{D}}bs^{\tiny\textcircled{D}}ca^{\tiny\textcircled{D}}&-a^{\tiny\textcircled{D}}bs^{\tiny\textcircled{D}}\\
-s^{\tiny\textcircled{D}}ca^{\tiny\textcircled{D}}&s^{\tiny\textcircled{D}}
\end{array}
\right).$$\end{cor}
\begin{proof} As in the proof of Theorem 3.4, we complete the proof by Theorem 3.8.\end{proof}

\begin{cor} Let $x$ be defined as in $(2)$. If $a\in \mathcal{A}^{\tiny\textcircled{\#}}, s\in \mathcal{A}^{-1}$ and $a^{\pi}b=0$, then
$x$ has a core inverse. In this case,
$$x^{\tiny\textcircled{\#}}=\left(
\begin{array}{cc}
a^{\tiny\textcircled{D}}+a^{\tiny\textcircled{D}}bs^{-1}ca^{\tiny\textcircled{D}}&-a^{\tiny\textcircled{D}}bs^{-1}\\
-s^{-1}ca^{\tiny\textcircled{D}}&s^{-1}
\end{array}
\right).$$\end{cor}
\begin{proof} In view of Corollary 3.9, we see that $$x^{\tiny\textcircled{D}}=\left(
\begin{array}{cc}
a^{\tiny\textcircled{\#}}+a^{\tiny\textcircled{\#}}bs^{-1}ca^{\tiny\textcircled{\#}}&-a^{\tiny\textcircled{D}}bs^{-1}\\
-s^{-1}ca^{\tiny\textcircled{\#}}&s^{-1}
\end{array}
\right).$$ Then we verify that
$$\begin{array}{rll}
(1-xx^{\tiny\textcircled{D}})x&=&\left(
\begin{array}{cc}
1-aa^{\tiny\textcircled{\#}}&0\\
0&0
\end{array}
\right)\left(
\begin{array}{cc}
a&b\\
c&d
\end{array}
\right)\\
&=&\left(
\begin{array}{cc}
a^{\sigma}a&a^{\sigma}b\\
0&0
\end{array}
\right)=0.
\end{array}$$ Therefore $x$ has core inverse and $x^{\tiny\textcircled{\#}}=x^{\tiny\textcircled{D}}$, as required.\end{proof}

The following result can now be established.

\begin{thm} Let $x$ be defined as in $(1)$. If $$a^{\pi}b=s^{\pi}ca^{\tiny\textcircled{d}}=0,$$ then $x$ has generalized core-EP inverse. In this case,
$$x^{\tiny\textcircled{d}}=\left(
\begin{array}{cc}
a^{\tiny\textcircled{d}}+a^{\tiny\textcircled{d}}bs^{\tiny\textcircled{d}}ca^{\tiny\textcircled{d}}&-a^{\tiny\textcircled{d}}bs^{\tiny\textcircled{d}}\\
-s^{\tiny\textcircled{d}}ca^{\tiny\textcircled{d}}&s^{\tiny\textcircled{d}}\end{array}
\right).$$\end{thm}
\begin{proof} Since $s^{\pi}ca^{\tiny\textcircled{d}}=0$, we deduce that
$$\begin{array}{rll}
d&=&s+ca^{\tiny\textcircled{d}}b\\
&=&s+(ss^{\tiny\textcircled{d}}ca^{\tiny\textcircled{d}})b\\
&=&s+(ss^{\tiny\textcircled{d}})(d-s)\\
&=&s^{\sigma}s+ss^{\tiny\textcircled{d}}d.
\end{array}$$ Then we have a decomposition $x=y+z$, where
$$y=\left(
\begin{array}{cc}
aa^{\sigma}&0\\
s^{\sigma}c&s^{\sigma}s
\end{array}
\right), z=\left(
\begin{array}{cc}
a^2a^{\tiny\textcircled{d}}&aa^{\tiny\textcircled{d}}b\\
ss^{\tiny\textcircled{d}}c&ss^{\tiny\textcircled{d}}d
\end{array}
\right).$$
Moreover, we have $z=z_1+z_2,$ where
$$z_1=\left(
\begin{array}{cc}
a^2a^{\tiny\textcircled{d}}&aa^{\tiny\textcircled{d}}bss^{\tiny\textcircled{d}}\\
ss^{\tiny\textcircled{d}}caa^{\tiny\textcircled{d}}&ss^{\tiny\textcircled{d}}dss^{\tiny\textcircled{d}}
\end{array}
\right), z_2=\left(
\begin{array}{cc}
0&aa^{\tiny\textcircled{d}}bs^{\sigma}\\
ss^{\tiny\textcircled{d}}ca^{\sigma}&ss^{\tiny\textcircled{d}}ds^{\sigma}
\end{array}
\right).$$

Then $x=y+z=z_1+(y+z_2)$.

Obviously, $z_2^2=0$, and then $z_2$ is nilpotent.

We further verify that $y+z_2$ is quasinilpotent and $(y+z_2)z_1=yz_1=y(z-z_2)=-yz_2=0.$

We claim that $z_1\in \mathcal{A}^{\tiny\textcircled{d}}$.

Clearly, $\big(a^2a^{\tiny\textcircled{d}}\big)^{\tiny\textcircled{\#}}=a^{\tiny\textcircled{d}}$ and
$\big(a^2a^{\tiny\textcircled{d}}\big)^{\sigma}=a^{\sigma}$. We compute that
$$\begin{array}{rll}
S&:=&ss^{\tiny\textcircled{d}}dss^{\tiny\textcircled{d}}-
(ss^{\tiny\textcircled{d}}caa^{\tiny\textcircled{d}})\big(a^2a^{\tiny\textcircled{d}}\big)^{\tiny\textcircled{d}}(aa^{\tiny\textcircled{d}}bss^{\tiny\textcircled{d}})\\
&=&ss^{\tiny\textcircled{d}}dss^{\tiny\textcircled{d}}-ss^{\tiny\textcircled{d}}ca(a^{\tiny\textcircled{d}})^2bss^{\tiny\textcircled{d}}\\
&=&ss^{\tiny\textcircled{d}}\big(d-ca^{\tiny\textcircled{d}}b\big)ss^{\tiny\textcircled{d}}\\
&=&s^2s^{\tiny\textcircled{d}}.
\end{array}$$ Hence, $S\in \mathcal{A}^{\tiny\textcircled{\#}}$ and $S^{\tiny\textcircled{\#}}=s^{\tiny\textcircled{d}}.$
Then $S^{\sigma}=1-SS^{\tiny\textcircled{\#}}=s^{\sigma}$.
We check that
$$\begin{array}{rll}
\big(a^2a^{\tiny\textcircled{d}}\big)^{\pi}\big(aa^{\tiny\textcircled{d}}bss^{\tiny\textcircled{d}}\big)&=&a^{\pi}bss^{\tiny\textcircled{d}}=0,\\
\big(ss^{\tiny\textcircled{d}}caa^{\tiny\textcircled{d}}\big)\big(a^2a^{\tiny\textcircled{d}}\big)^{\sigma}&=&
\big(ss^{\tiny\textcircled{d}}caa^{\tiny\textcircled{d}}\big)a^{\sigma}=0,\\
\big(aa^{\tiny\textcircled{d}}bss^{\tiny\textcircled{d}}\big)S^{\sigma}&=&\big(aa^{\tiny\textcircled{d}}bss^{\tiny\textcircled{d}}\big)s^{\sigma}=0.\\
S^{\pi}\big(ss^{\tiny\textcircled{d}}caa^{\tiny\textcircled{d}}\big)&=&s^{\pi}ss^{\tiny\textcircled{d}}caa^{\tiny\textcircled{d}}=0.\end{array}$$
In light of Corollary 3.7, $z_1\in M_2(\mathcal{A})^{\tiny\textcircled{\#}}$ and

$$z_1^{\tiny\textcircled{\#}}=\left(
\begin{array}{cc}
a^{\tiny\textcircled{d}}+a^{\tiny\textcircled{d}}bs^{\tiny\textcircled{d}}ca^{\tiny\textcircled{d}}&-a^{\tiny\textcircled{d}}bs^{\tiny\textcircled{d}}\\
-s^{\tiny\textcircled{d}}ca^{\tiny\textcircled{d}}&s^{\tiny\textcircled{d}}\end{array}
\right).$$
By virtue of ~\cite[Theorem 2.1]{C3}, $x$ has a generalized core-EP inverse and $x^{\tiny\textcircled{d}}=z_1^{\tiny\textcircled{\#}}$, as asserted.\end{proof}

\begin{cor} Let $x$ be defined as in $(1)$. If $$a^{\pi}b=s^{\pi}ca^{\tiny\textcircled{d}}=0,$$ then $x$ has core inverse if and only if $a,s\in \mathcal{A}^{\tiny\textcircled{\#}}$ and $s^{\pi}c=0$. In this case,
$$x^{\tiny\textcircled{\#}}=\left(
\begin{array}{cc}
a^{\tiny\textcircled{\#}}+a^{\tiny\textcircled{\#}}bs^{\tiny\textcircled{\#}}ca^{\tiny\textcircled{\#}}&-a^{\tiny\textcircled{\#}}bs^{\tiny\textcircled{\#}}\\
-s^{\tiny\textcircled{\#}}ca^{\tiny\textcircled{\#}}&s^{\tiny\textcircled{\#}}\end{array}
\right).$$\end{cor}
\begin{proof} By virtue of Theorem 3.1, $x$ has generalized core-EP inverse. In this case,
$$x^{\tiny\textcircled{d}}=\left(
\begin{array}{cc}
a^{\tiny\textcircled{d}}+a^{\tiny\textcircled{d}}bs^{\tiny\textcircled{d}}ca^{\tiny\textcircled{d}}&-a^{\tiny\textcircled{d}}bs^{\tiny\textcircled{d}}\\
-s^{\tiny\textcircled{d}}ca^{\tiny\textcircled{d}}&s^{\tiny\textcircled{d}}\end{array}
\right).$$ We directly check that

$$\begin{array}{rll}
xx^{\tiny\textcircled{d}}&=&\left(
\begin{array}{cc}
a&b\\
c&d
\end{array}
\right)\left(
\begin{array}{cc}
a^{\tiny\textcircled{d}}+a^{\tiny\textcircled{d}}bs^{\tiny\textcircled{d}}ca^{\tiny\textcircled{d}}&-a^{\tiny\textcircled{d}}bs^{\tiny\textcircled{d}}\\
-s^{\tiny\textcircled{d}}ca^{\tiny\textcircled{d}}&s^{\tiny\textcircled{d}}\end{array}
\right)\\
&=&\left(
\begin{array}{cc}
1-aa^{\tiny\textcircled{d}}&0\\
0&1-ss^{\tiny\textcircled{d}}
\end{array}
\right).
\end{array}$$ Hence, $xx^{\tiny\textcircled{d}}x=x$ if and only if
$$\begin{array}{rll}
\left(
\begin{array}{cc}
a^{\sigma}a&a^{\sigma}b\\
s^{\sigma}c&\sigma d
\end{array}
\right).
\end{array}$$ Hence, $xx^{\tiny\textcircled{d}}x=x$ if and only if $a,s\in \mathcal{A}^{\tiny\textcircled{\#}}$ and
$s^{\pi}c=0$.
Therefore we complete the proof by Lemma 2.4.\end{proof}

\begin{thm} Let $x$ be defined as in $(1)$. If $$a^{\pi}b=s^{\pi}b^*=ss^{\sigma}ca^{\tiny\textcircled{d}}=0,$$ then $x$ has generalized core-EP inverse. In this case, $$x^{\tiny\textcircled{d}}=\left(
\begin{array}{cc}
a^{\tiny\textcircled{d}}+a^{\tiny\textcircled{d}}bs^{\tiny\textcircled{d}}ca^{\tiny\textcircled{d}}&-a^{\tiny\textcircled{d}}bs^{\tiny\textcircled{d}}\\
-s^{\tiny\textcircled{d}}ca^{\tiny\textcircled{d}}&s^{\tiny\textcircled{d}}
\end{array}
\right).$$\end{thm}
\begin{proof} Since $bs^{\sigma}=0$, we see that $$\begin{array}{rll}
ss^{\sigma}+dss^{\tiny\textcircled{d}}&=&s[1-ss^{\tiny\textcircled{d}}]+[s+ca^{\tiny\textcircled{d}}b]ss^{\tiny\textcircled{d}}\\
&=&s+ca^{\tiny\textcircled{d}}(bss^{\tiny\textcircled{d}})\\
&=&s+ca^{\tiny\textcircled{d}}b=d.
\end{array}$$ Write $x=y+z$, where
$$y=\left(
\begin{array}{cc}
aa^{\sigma}&0\\
ca^{\sigma}&ss^{\sigma}
\end{array}
\right), z=\left(
\begin{array}{cc}
a^2a^{\tiny\textcircled{d}}&bss^{\tiny\textcircled{d}}\\
caa^{\tiny\textcircled{d}}&dss^{\tiny\textcircled{d}}
\end{array}
\right).$$ Clearly, $a^{\sigma}b=0$ and $$ss^{\sigma}dss^{\tiny\textcircled{d}}=ss^{\sigma}
[s-ca^{\tiny\textcircled{d}}b]ss^{\tiny\textcircled{d}}=-ss^{\sigma}ca^{\tiny\textcircled{d}}bss^{\sigma}=0.$$
Then $yz=0$ and $y$ is quasinilpotent.

As in the proof of Theorem 3.11, $z$ has core inverse and $$z^{\tiny\textcircled{\#}}=\left(
\begin{array}{cc}
a^{\tiny\textcircled{d}}+a^{\tiny\textcircled{d}}bs^{\tiny\textcircled{d}}ca^{\tiny\textcircled{d}}&-a^{\tiny\textcircled{d}}bs^{\tiny\textcircled{d}}\\
-s^{\tiny\textcircled{d}}ca^{\tiny\textcircled{d}}&s^{\tiny\textcircled{d}}
\end{array}
\right).$$ According to~\cite[Theorem 2.1]{C3}, we derive that
 $x$ has generalized core-EP inverse and $x^{\tiny\textcircled{d}}=z^{\tiny\textcircled{\#}},$ as asserted.\end{proof}

\begin{cor} Let $x$ be defined as in $(1)$. If $a,s\in \mathcal{A}^{\tiny\textcircled{\#}}$ and $a^{\pi}b=s^{\pi}b^*=s^{\pi}c=0,$ then $x$ has core inverse. In this case,
$$x^{\tiny\textcircled{\#}}=\left(
\begin{array}{cc}
a^{\tiny\textcircled{\#}}+a^{\tiny\textcircled{\#}}bs^{\tiny\textcircled{\#}}ca^{\tiny\textcircled{\#}}&-a^{\tiny\textcircled{\#}}bs^{\tiny\textcircled{\#}}\\
-s^{\tiny\textcircled{\#}}ca^{\tiny\textcircled{\#}}&s^{\tiny\textcircled{\#}}
\end{array}
\right).$$\end{cor}
\begin{proof} Straightforward by Theorem 3.13.\end{proof}

\section{numerical examples}

We take the involution on ${\Bbb C}^{n\times n}$ as the conjugate transpose. We now illustrate Theorem 3.4 and Corollary 3.7 by the following numerical examples.

\begin{exam}\end{exam} Let the $8 \times 8$ matrix $x$ be given by:
\[
x = \left(
\begin{array}{cc}
a & b \\
c & d
\end{array}
\right) = \left(
\begin{array}{cccc|cccc}
0&1&0&0&1&0&1&0\\
1&0&0&0&0&1&0&1\\
0&0&0&1&0&0&0&0\\
0&0&0&0&0&0&0&0\\
\hline
1&0&0&0&0&0&1&0\\
0&1&0&0&0&0&0&1\\
1&0&0&0&0&0&0&0\\
0&1&0&0&0&0&0&0\\
\end{array}
\right)\in {\Bbb C}^{8\times 8}.
\]

Then $a^{\tiny\textcircled{D}}=\left(
\begin{array}{cccc}
0&1&0&0\\
1&0&0&0\\
0&0&0&0\\
0&0&0&0
\end{array}
\right)$ and $s:=d-ca^{\tiny\textcircled{D}}b=\left(
\begin{array}{cccc}
0&-1&1&-1\\
-1&0&-1&1\\
0&-1&0&-1\\
-1&0&-1&0
\end{array}
\right)$.

We compute that $$s^{-1} = \begin{pmatrix}
-1 & 0 & 1 & -1 \\
0 & -1 & -1 & 1 \\
1 & 0 & -1 & 0 \\
0 & 1 & 0 & -1
\end{pmatrix} $$

Moreover, we verify that $$\begin{array}{rll}
a^{\pi}b&=&\begin{pmatrix}
0 & 0 & 0 & 0 \\
0 & 0 & 0 & 0 \\
0 & 0 & 1 & 0 \\
0 & 0 & 0 & 1
\end{pmatrix}\begin{pmatrix}
1 & 0 & 1 & 0 \\
0 & 1 & 0 & 1 \\
0 & 0 & 0 & 0 \\
0 & 0 & 0 & 0
\end{pmatrix}=0, \\
ca^{\tau}&=&\begin{pmatrix}
1 & 0 & 0 & 0 \\
0 & 1 & 0 & 0 \\
1 & 0 & 0 & 0 \\
0 & 1 & 0 & 0
\end{pmatrix}\begin{pmatrix}
0 & 0 & 0 & 0 \\
0 & 0 & 0 & 0 \\
0 & 0 & 1 & 0 \\
0 & 0 & 0 & 1
\end{pmatrix}=0.
\end{array}$$

By virtue of Corollary 3.12, we obtain $$\begin{array}{rll}
x^{\tiny\textcircled{D}}&=&\left(
\begin{array}{cc}
a^{\tiny\textcircled{D}}+a^{\tiny\textcircled{D}}bs^{-1}ca^{\tiny\textcircled{D}}&-a^{\tiny\textcircled{D}}bs^{-1}\\
-s^{-1}ca^{\tiny\textcircled{D}}&s^{-1}
\end{array}
\right)\\
&=&\left(
\begin{array}{cccc|cccc}
0 & 0 & 0 & 0 & 0 & 0 & 1 & 0 \\
0 & 0 & 0 & 0 & 0 & 0 & 0 & 1 \\
0 & 0 & 0 & 0 & 0 & 0 & 0 & 0 \\
0 & 0 & 0 & 0 & 0 & 0 & 0 & 0 \\
\hline
1 & 0 & 0 & 0 & -1 & 0 & 1 & -1 \\
0 & 1 & 0 & 0 & 0 & -1 & -1 & 1 \\
0 & 0 & 0 & 0 & 1 & 0 & -1 & 0 \\
0 & 0 & 0 & 0 & 0 & 1 & 0 & -1
\end{array}
\right).
\end{array}$$ \\

Let $H$ be a Hilbert space. Then the algebra $\mathcal{B}(H)$ of bounded linear operators on $H$ is a Banach *-algebra with the adjoint as its involution. The space ${\ell}^2({\Bbb N})$ is a Hilbert space consisting of all square-summable infinite sequences of complex numbers. To illustrate Theorem 3.11, we present an block operator matrix that has a generalized core inverse.

\begin{exam}\end{exam} Let $\mathcal{H} = \ell^2(\mathbb{N})$ with the standard orthonormal basis $\{e_n\}_{n=1}^\infty$. Define the bounded linear operator $a: \mathcal{H} \to \mathcal{H}$ by:
\begin{align*}
    a e_1 &= e_1, \\
    a e_2 &= e_1, \\
    a e_n &= \frac{1}{n-2} e_n, \quad n \ge 3.
\end{align*}
Equivalently, $a$ can be written as a block diagonal operator:
\[ a = \begin{pmatrix} 1 & 1 \\ 0 & 0 \end{pmatrix} \oplus \operatorname{diag}\left(1, \frac{1}{2}, \frac{1}{3}, \dots\right). \]

The core inverse of $a$, denoted $x = a^{\tiny\textcircled{\#}}$, is defined by:
\begin{align*}
    x e_1 &= e_1, \\
    x e_2 &= 0, \\
    x e_n &= (n-2) e_n, \quad n \ge 3.
\end{align*}
Equivalently:
\[ x = \begin{pmatrix} 1 & 0 \\ 0 & 0 \end{pmatrix} \oplus \operatorname{diag}(1, 2, 3, \dots). \]

We now verify the three core inverse conditions:

$(1)$ For $n \ge 3$, $xa^2 e_n = x\left(\frac{1}{(n-2)^2} e_n\right) = \frac{1}{n-2} e_n = a e_n$. \\
For the basis vectors $e_1$ and $e_2$: $xa^2 e_1 = x(e_1) = e_1 = a e_1$ and $xa^2 e_2 = x(e_1) = e_1 = a e_2$. Thus, $xa^2 = a$.

$(2)$ For $n \ge 3$, $ax^2 e_n = a\left((n-2)^2 e_n\right) = (n-2) e_n = x e_n$. \\
For the other basis vectors: $ax^2 e_1 = a(e_1) = e_1 = x e_1$ and $ax^2 e_2 = a(0) = 0 = x e_2$. Thus, $ax^2 = x$.

$(3)$ We compute the action of $ax$: $ax e_1 = e_1$, $ax e_2 = 0$, and $ax e_n = e_n$ for $n \ge 3$. \\
Therefore, $ax$ is exactly the orthogonal projection. Since orthogonal projections are self-adjoint, $(ax)^* = ax$.

Let $S$ be the unilateral shift ($Se_n = e_{n+1}$). Let
\begin{align*}
b &= S^2 + e_1 \otimes e_1, \\
c &= S^* + e_3 \otimes e_1.
\end{align*}
where $S^* e_1 = 0$, $S^* e_n = e_{n-1}$ ($n \ge 2$), and $u \otimes v$ denotes the rank-1 operator $w \mapsto \langle w, v \rangle u$.
Explicitly on the basis:
\begin{align*}
b e_1 &= e_1, & b e_2 &= e_4, & b e_n &= e_{n+2}, \quad n \ge 3, \\
c e_1 &= e_3, & c e_2 &= 0, & c e_n &= e_{n-1}, \quad n \ge 3.
\end{align*}

Then we verify the following equalities:
\begin{align*}
(I - a a^{\tiny\textcircled{\#}}) b &= 0, \\
c (I - a a^{\tiny\textcircled{\#}}) &= 0.
\end{align*}

Choose $d=1-ca^{\tiny\textcircled{\#}}b$. Thus the generalized Schur complement $s = d - c a^{\tiny\textcircled{\#}} b = I$, the identity operator, which is obviously core invertible, with $s^{\tiny\textcircled{\#}} = I$ and $ss^{\tiny\textcircled{\#}} = I$.
Since $s = I$, we have $s s^{\tiny\textcircled{\#}} = I$, so $I - s s^{\tiny\textcircled{\#}} = 0$, hence $(I - s s^{\tiny\textcircled{\#}}) c = 0$ and $b (I - s s^{\tiny\textcircled{\#}}) = 0$.

We consider the $2\times2$ operator matrix
\[
M = \begin{pmatrix} a & b \\ c & d \end{pmatrix} \in \mathcal{B}(\mathcal{H} \oplus \mathcal{H})
\]
with the generalized Schur complement $s = d - c a^{\tiny\textcircled{\#}} b$. By virtue of Corollary 3.7, the operator matrix $M$ has a core invertible, with core inverse given by the block formula:
\[
M^{\tiny\textcircled{\#}}= \begin{pmatrix}
a^{\tiny\textcircled{\#}} + a^{\tiny\textcircled{\#}} bc a^{\tiny\textcircled{\#}} & -a^{\tiny\textcircled{\#}} b \\
-c a^{\tiny\textcircled{\#}} & I
\end{pmatrix}.
\]\\

\begin{exam}\end{exam} Let the $4 \times 4$ matrix $x$ is given by:
\[
x = \left(
\begin{array}{cc}
a & b \\
c & d
\end{array}
\right) = \left(
\begin{array}{cc|cc}
2 & 0 & 1 & 2 \\
0 & 0 & 0 & 0 \\
\hline
1 & i & 3/2 & 2 \\
0 & 0 & 1 & 2
\end{array}
\right).
\]

Since $a$ is a diagonal matrix with $\operatorname{rank}(a)=1$, its group inverse and Moore-Penrose inverse are easily obtained:
\[
a^{\#} = \begin{pmatrix} 1/2 & 0 \\ 0 & 0 \end{pmatrix}, \quad a^{\dag} = \begin{pmatrix} 1/2 & 0 \\ 0 & 0 \end{pmatrix}
\]
The core inverse $a^{\tiny\textcircled{\#}}$ is given by $a^{\#}aa^{\dag}$:
\[
a^{\tiny\textcircled{\#}} = \begin{pmatrix} 1/2 & 0 \\ 0 & 0 \end{pmatrix} \begin{pmatrix} 2 & 0 \\ 0 & 0 \end{pmatrix} \begin{pmatrix} 1/2 & 0 \\ 0 & 0 \end{pmatrix} = \begin{pmatrix} 1/2 & 0 \\ 0 & 0 \end{pmatrix} = a^{\#}
\]

Compute $aa^{\#}$:
\[
aa^{\#} = \begin{pmatrix} 2 & 0 \\ 0 & 0 \end{pmatrix} \begin{pmatrix} 1/2 & 0 \\ 0 & 0 \end{pmatrix} = \begin{pmatrix} 1 & 0 \\ 0 & 0 \end{pmatrix}
\]
Compute $I_2 - aa^{\#}$:
\[
I_2 - aa^{\#} = \begin{pmatrix} 1 & 0 \\ 0 & 1 \end{pmatrix} - \begin{pmatrix} 1 & 0 \\ 0 & 0 \end{pmatrix} = \begin{pmatrix} 0 & 0 \\ 0 & 1 \end{pmatrix}
\]
Multiply by $b$, we derive
\[
(I_2 - aa^{\#})b = \begin{pmatrix} 0 & 0 \\ 0 & 1 \end{pmatrix} \begin{pmatrix} 1 & 2 \\ 0 & 0 \end{pmatrix} = \begin{pmatrix} 0 & 0 \\ 0 & 0 \end{pmatrix} = 0
\]

Compute $ca^{\tiny\textcircled{\#}}b$:
\[
ca^{\tiny\textcircled{\#}}b = \begin{pmatrix} 1 & i \\ 0 & 0 \end{pmatrix} \begin{pmatrix} 1/2 & 0 \\ 0 & 0 \end{pmatrix} \begin{pmatrix} 1 & 2 \\ 0 & 0 \end{pmatrix} = \begin{pmatrix} 1/2 & 0 \\ 0 & 0 \end{pmatrix} \begin{pmatrix} 1 & 2 \\ 0 & 0 \end{pmatrix} = \begin{pmatrix} 1/2 & 1 \\ 0 & 0 \end{pmatrix}
\]
Compute $s = d - ca^{\tiny\textcircled{\#}}b$:
\[
s = \begin{pmatrix} 3/2 & 2 \\ 1 & 2 \end{pmatrix} - \begin{pmatrix} 1/2 & 1 \\ 0 & 0 \end{pmatrix} = \begin{pmatrix} 1 & 1 \\ 1 & 2 \end{pmatrix}
\]
The matrix $s$ is invertible since its determinant is $\det(s) = 2 - 1 = 1 \neq 0$. For an invertible matrix, the Drazin inverse is exactly its standard inverse, i.e., $s^D = s^{-1}$.
\[
s^D = s^{-1} = \begin{pmatrix} 2 & -1 \\ -1 & 1 \end{pmatrix}
\]

We easily check the condition: $(I_2 - ss^D)ca = 0$. By virtue of Theorem 3.11,
$$\begin{array}{rll}
x^{\tiny\textcircled{\#}}&=&\left(
\begin{array}{cc}
a^{\tiny\textcircled{\#}}+a^{\tiny\textcircled{\#}}bs^{-1}ca^{-1}&-a^{\tiny\textcircled{\#}}bs^{-1}\\
-s^{-1}ca^{\tiny\textcircled{\#}}&s^{-1}\end{array}
\right)\\
&=&\left(
\begin{array}{cc|cc}
1/2 & 0 & 0 & -1/2 \\
0 & 0 & 0 & 0 \\
\hline
-1 & 0 & 2 & -1 \\
1/2 & 0 & -1 & 1\\
\end{array}
\right).
\end{array}$$
Evidently, we verify that $x^{\tiny\textcircled{\#}}=x(x^{\tiny\textcircled{\#}})^2, (xx^{\tiny\textcircled{\#}})^*=xx^{\tiny\textcircled{\#}}$ and $x=x^{\tiny\textcircled{\#}}x^2$.

\end{document}